\documentclass[12pt]{article}
\usepackage{amsmath, amssymb}
\usepackage{geometry}
\usepackage{setspace}
\usepackage{hyperref}
\usepackage{graphicx}
\usepackage{times}
\usepackage{xcolor}
\newtheorem{lemma}{Lemma}

\newtheorem{prop}{Proposition}

\usepackage[numbers]{natbib}
\usepackage{subfigure}

\newenvironment{proof}{\paragraph{Proof:}}{\hfill$\square$}
\title{Advance versus Spot Selling in Queueing Systems: the Effect of Uncertainty Sources}
\author{
  Yumeng Liu$^{a}$ \and Dongyuan Zhan$^{b,}$\thanks{Corresponding author. Email: \texttt{dzhan@ustc.edu.cn}}  
}

\begin{document}
\date{}
\maketitle

\begin{center}
\textit{
$^a$ International Institute of Finance, School of Management, University of Science and Technology of China, Hefei 230026, China \\ \texttt{ymliu@ustc.edu.cn} \\
$^b$ School of Management, Faculty of Business for Science and Technology, University of Science and Technology of China, Hefei 230026, China \\
\texttt{dzhan@ustc.edu.cn}
}
\end{center}

\bigskip


\begin{abstract}
We develop an unobservable M/M/1 queueing model to analyze advance selling and spot selling under uncertainty. Customers are heterogeneous in waiting costs, drawn from a uniform distribution. We disentangle two fundamental sources of uncertainty: customer valuation and service capacity. We derive the respective revenue functions for both advance and spot selling strategies and compare their performance. Our analysis reveals that advance selling is more profitable when customer valuations are uncertain. Conversely, spot selling emerges as the more profitable strategy when uncertainty originates from service capacity. When both sources are present, our numerical experiments suggest that the preferred regime depends on the relative dispersion of the two uncertainties: advance selling tends to be favored when valuation uncertainty is sufficiently pronounced relative to capacity uncertainty, whereas spot selling tends to be favored when capacity uncertainty is more pronounced.  These results highlight how the origin and magnitude of uncertainty shape the optimal selling strategy in service queues and offer practical guidance for platforms and service providers deciding between advance and spot selling. 
\end{abstract}


\section{Introduction}
Advance selling, a strategy in which firms offer products or services for purchase ahead of the consumption period, is widely adopted across various service industries. Sectors such as theme parks, gyms, and museums leverage this approach to proactively manage demand and secure revenue streams. In these contexts, customers often face congestion and waiting times, which decrease their overall utility. Additionally, customers may encounter significant uncertainty regarding factors that can profoundly impact their service experience and utility. This uncertainty stems primarily from two key sources: fluctuations in customer valuations of the service and variability in the service capacity provided. 

In the \emph{uncertain valuation} scenario, customers' valuation of the service is not fixed beforehand. For example, the value derived from a visit to a theme park can vary significantly depending on external factors such as weather conditions. A sunny day might result in a high valuation, while a rainy day could lead to a lower valuation. Consequently, in advance selling, each customer faces uncertainty about their eventual valuation of the service. Another critical source of uncertainty is \emph{uncertain capacity}. For instance, theme parks may need to temporarily close certain facilities due to adverse weather conditions. Also, in the expanding gig economy, service capacity can be inherently stochastic, depending on the number of servers or volunteers available. As a result, the rate at which services are delivered can fluctuate significantly. We model capacity uncertainty through an uncertain service rate in an M/M/1 queue. This should be interpreted as a reduced-form representation of uncertainty in the effective processing speed.  Although some motivating examples, such as facility closures or staffing changes, may be more naturally captured by multi-server queues, the single-server stochastic-service-rate model provides a tractable benchmark for isolating the effect of supply-side uncertainty. 

In contrast, under spot selling, both types of uncertainty are resolved at the time of purchase, as customers make their decisions after observing their realized valuation and/or the available service capacity. If the valuation or the capacity is high, more customers purchase; otherwise, fewer customers purchase. 

When the seller must commit to a single price, the optimal selling mechanism is far from self-evident. Advance selling locks customers into the bad scenario (e.g., low valuations or poor service conditions), but it also restricts the ability to attract additional customers in the good scenario. This trade-off underscores the complexity of choosing the right pricing strategy. We endeavor to analyze queueing systems in which customers experience waiting costs and make economically rational purchase decisions based on their expectations of waiting times and service value. Specifically, we formulate and rigorously compare two distinct pricing strategies---advance selling and spot selling---under the aforementioned uncertainty regimes. 

Our main findings are simple and sharp: (i) under valuation uncertainty only, advance selling weakly dominates spot selling; (ii) under capacity uncertainty only, spot selling weakly dominates advance selling; (iii) when both are present, our numerical experiments suggest that the preferred regime depends on the relative dispersion of the two uncertainties. These key insights underscore the importance for service providers to carefully assess the nature and source of the uncertainty affecting their operations when deciding between advance and spot selling mechanisms.


\subsection*{Related Literature}

Our work connects two large literatures: queueing games and advance selling. On the queueing game side, the classic observable-queue model of \cite{Naor1969} and its unobservable counterpart in \cite{Edelson1975} launched the study of customers’ strategic join/balk behavior. Comprehensive treatments appear in \cite{HassinHaviv2003} and \cite{Hassin2016}. Building on these foundations, a substantial stream examines pricing and priority design in queues; see \cite{MendelsonWhang1990} on incentive-compatible priority pricing, \cite{AfecheMendelson2004} on priority auctions and pricing with generalized delay costs, and \cite{Afeche2013} on price/lead-time menu screening where ``strategic delay'' can be revenue-optimal. A complementary stream studies information disclosure and announcements in service systems, e.g., delay information in contact centers \cite{ArmonyShimkinWhitt2009} and the performance implications of partial vs.\ full delay information \cite{GuoZipkin2007}.  More recently, Bayesian persuasion has been adopted to design queue-length disclosure under uncertain service quality \cite{Guo2022, Guo2023signaling}. 

On the advance selling side, the marketing literature emphasizes valuation uncertainty as the economic driver of preorders and advance purchases. The seminal analysis of \cite{XieShugan2001} (and related work by \cite{ShuganXie2005}) shows when and how advance selling improves profit by separating purchase and consumption. \cite{FayXie2010} compares advance selling with probabilistic selling. In operations, advance selling and preorders are used to extract advance demand information and price discriminate among strategic consumers; see \cite{LiZhang2013} and follow-ons such as \cite{Wei2018}. Mechanisms adjacent to advance selling include binding reservations \cite{OsadchiyVulcano2010} and advance reservations with priority or preemption \cite{ChamberlainSimhonStarobinski2021}, both highlighting strategic timing and rationing effects.

Our paper sits at the intersection: we embed advance vs.\ spot selling in an unobservable $M/M/1$ queue with endogenous waiting costs, and compare revenues across distinct uncertainty sources (valuation vs.\ capacity). Relative to \cite{XieShugan2001}, who abstract from queueing delays, we explicitly model waiting costs and strategic waiting; compared with price/lead-time menu and priority-pricing papers (e.g., \cite{MendelsonWhang1990,Afeche2013}) we focus on single-price mechanisms that map cleanly to advance vs.\ spot regimes. Finally, while \cite{Wang2024} analyzes advance selling and service upgrades in a priority queue under demand uncertainty, we use a simpler single-server model without priority to isolate how the \emph{source} of uncertainty—valuation or capacity—shifts the revenue ranking between advance and spot selling.

\section{Model and Analysis}

We consider a service system modeled as an unobservable M/M/1 queue. Potential customers arrive according to a Poisson process with rate $\lambda$. They have homogeneous service valuation $v$, and heterogeneous waiting costs $c$, uniformly distributed on the interval $[0,1]$. The expected utility of purchasing is $v-p-cW$, where $p$ is the price and $W$ is the expected sojourn time. 

The seller aims to maximize revenue by setting a single price $p$. Each customer decides to purchase the service if their valuation exceeds the sum of the price and the waiting cost incurred. For a given price \(p\), we focus on symmetric cutoff equilibria. A cutoff equilibrium is a threshold \(T\in[0,1]\) such that customers with waiting-cost type \(c\le T\) join and customers with \(c>T\) balk. Given a conjectured cutoff $T$, the effective arrival rate is \(\lambda T\), and all joining customers face the expected sojourn time \(W(T)=1/(\mu-\lambda T)\) as long as $\mu>\lambda T$. Since the utility \(v-p-cW(T)\) is decreasing in \(c\), any best response has a cutoff form. The marginal customer \(c=T\) is indifferent, which gives
\begin{equation}
v - p - T\,W(T)=0, 
\label{eq:adv_val}
\end{equation}
where $v$ is the service valuation, assumed to be homogeneous across customers in this benchmark. For a given price $p$, we can solve equation (\ref{eq:adv_val}) for the equilibrium cutoff:
\begin{equation} \label{eq:T}
    T(p,v) := \frac{\mu(v-p)}{\lambda(v-p)+1}.
\end{equation}

Throughout the paper, we focus on service systems with scarce capacity. Specifically, we assume $\lambda \ge \mu$. This assumption is natural for congested service systems in which the potential demand is larger than the available service capacity, and only a fraction of potential customers can be served. In such settings, customer waiting costs play an important role in regulating demand: only customers with sufficiently low waiting costs choose to join the system, while others balk. 
Because \(\lambda\ge \mu\), this cutoff satisfies $T(p,v,\mu)<\frac{\mu}{\lambda}\le 1$, which is an interior point. The expected revenue rate for a given valuation $v$ and price $p<v$ is then:
\begin{equation}
R_p(p, v)= \lambda\,p\,T(p,v).
\label{eq:RevD}
\end{equation}
Since $R_p$ is concave in $p$, we can find the optimal price by solving the first-order condition $\frac{dR_p(p)}{dp} = 0$. This yields the optimal price:
\begin{equation} \label{OptP}
p^*(v) = \frac{1 + \lambda v - \sqrt{1 + \lambda v}}{\lambda} < v.
\end{equation}
Substituting $p^*(v)$ back into the cutoff and revenue function, we obtain the optimal cutoff and the maximum revenue:
\[
T(p^*(v),v) = \frac{\mu}{\lambda}\bigg(1-\frac{1}{\sqrt{1+\lambda v}}\bigg), \quad R^*(v) = \frac{\mu(\sqrt{1 + \lambda v}-1)^2}{\lambda}. 
\]

\textbf{Remark:}  
The uniform distribution is a normalized benchmark for heterogeneous delay sensitivity. 
This uniform assumption implies that the cutoff type \(T\) is also the fraction of customers who join, so the effective arrival rate is \(\lambda T\). This feature yields closed-form equilibrium and revenue expressions. Under a general waiting-cost distribution \(G\), the effective arrival rate would be \(\lambda G(T)\), and the equilibrium condition would become
\[
v-p-\frac{T}{\mu-\lambda G(T)}=0.
\]
This formulation can be used for numerical analysis, but closed-form dominance comparisons are generally not available. Consistent with several prior works in the literature (e.g., \cite{Zhang2007,Gilland2009,Yu2012}), we therefore adopt $c\sim U[0,1]$ to preserve analytical tractability. 


We call the above the base model. In the subsequent sections, we introduce uncertainty into the model, considering uncertainty in valuation and service capacity respectively.  For each source of uncertainty, we will analyze both advance selling and spot selling strategies and compare their profitability.  Under both strategies, the seller sets \emph{a single price} and cannot change prices based on the state. The key distinction lies in the timing of customer purchase decisions relative to the resolution of uncertainty.

\begin{itemize}
 \item \textbf{Advance Selling:} A single price $p_a$ is set before the uncertainty is resolved. Customers decide whether to purchase \emph{before} the realized state of uncertainty.
\item \textbf{Spot Selling:} A single price $p_s$ is set before the uncertainty is resolved. Customers make their purchase decisions \emph{after} they observe the realized state. 
\end{itemize}

\subsection{Uncertain Valuation}
We first consider the case where customer valuation is uncertain and follows a two-point distribution:
\begin{equation} \label{DisV}
v = \begin{cases}
v_h, & \text{with probability } \alpha,\\[1mm]
v_l, & \text{with probability } 1-\alpha,
\end{cases}
\end{equation}
where $v_h>v_l>0$, $0<\alpha<1$. Let $\bar{v}=\alpha v_h + (1-\alpha) v_l$ denote the expected service valuation. 

\paragraph{Spot Selling.}
Under spot selling, the seller sets a spot price $p_s$, and the valuation uncertainty is resolved \emph{before} customers make their purchase decisions.  Based on  (\ref{eq:T}), when customers observe a high valuation $v_h$, the marginal customer type is $[T(p_s,v_h)]^+$, where $[x]^+:=\max\{x,0\}$; when customers observe a low valuation $v_l$, the marginal customer type is $[T(p_s, v_l)]^+$, illustrated in Figure~\ref{fig:SpotS}. The expected revenue rate under spot selling is given by:
\begin{equation*}
R_s(p_s) = \lambda\, p_s\big(\alpha\,[T(p_s, v_h)]^+ + (1-\alpha)\,[T(p_s, v_l)]^+\big) = \alpha [R_p(p_s, v_h)]^+ + (1-\alpha)[R_p(p_s, v_l)]^+.
\label{eq:RS_model}
\end{equation*}
\begin{figure}[hbtp]
 \centering
 \includegraphics[scale=0.36]{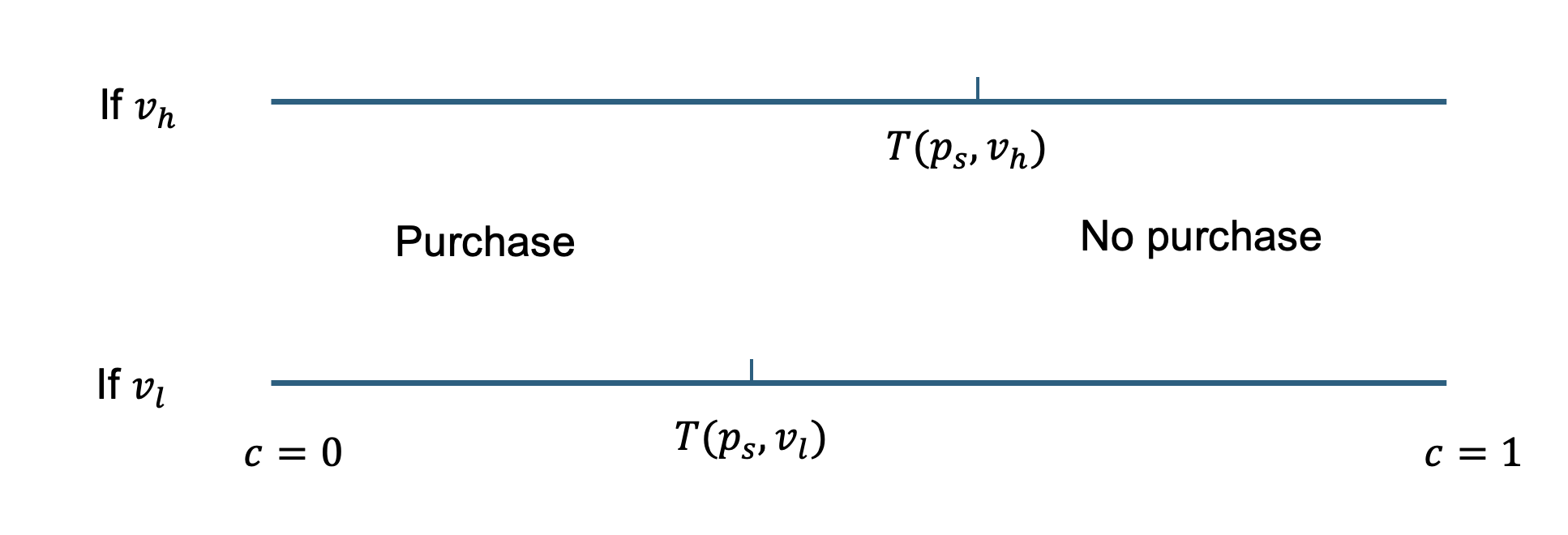}
\caption{Customer Behavior under Spot Selling}
\label{fig:SpotS}
\end{figure}

To find the optimal spot selling price, we need to maximize $R_s(p_s)$. Lemma~\ref{lemma:OptSpot} characterizes the optimal price structure.

\begin{lemma}\label{lemma:OptSpot}
Assume service valuation follows the two-point distribution in (\ref{DisV}). Under spot selling, if $p^*(v_h) = \frac{1 + \lambda v_h - \sqrt{1 + \lambda v_h}}{\lambda} \ge v_l$, then there exists a threshold probability $\hat{\alpha}\in(0,1)$ such that if $\alpha>\hat{\alpha}$, the optimal revenue is achieved at price $p^*(v_h)$. In all other cases, the optimal revenue is achieved at price $p_1^*\in(0,v_l)$, where $p_1^*$ solves $\frac{dR_s(p_s)}{dp_s} = 0$ on the interval $[0, v_l]$.
\end{lemma}

\begin{proof}
We analyze $R_s(p_s)$ by considering two different price intervals.
\begin{itemize}
\item \textbf{Left interval: $p_s\in[0,v_l]$}. $R_s(p_s)= R_{s1}(p_s):= \alpha R_p(p_s, v_h) + (1-\alpha)R_p(p_s, v_l)$ is a convex combination of two concave functions, and thus is also concave.  
\item \textbf{Right interval: $p_s\in[v_l,v_h]$}. Since $p_s \ge v_l$, we have $T(p_s, v_l) \le 0$ and thus $R_p(p_s, v_l) \le 0$.  Therefore, $R_s(p_s) = R_{s2}(p_s):= \alpha R_p(p_s, v_h)$, which is concave on $[v_l, v_h]$. 
\end{itemize}
On the left interval, $R_p(p_s,v_l)\ge 0$ so $R_{s1}(p_s)\ge R_{s2}(p_s)$; on the right interval, $R_p(p_s,v_l)\le 0$ so $R_{s1}(p_s)\le R_{s2}(p_s)$. Therefore, we can extend the supports of $R_{s1}$ and $R_{s2}$ without affecting the maximum of $R_{s}$:
\[
\max_{p_s\in[0,v_h]} R_s(p_s) = \max_{p_s\in[0,v_h]} \max\{R_{s1}(p_s), R_{s2}(p_s)\} = \max\{\max_{p_s\in[0,v_h]}R_{s1}(p_s),\max_{p_s\in[0,v_h]}R_{s2}(p_s)\}
\]
Denote by $p_1^*$ the maximizer for $R_{s1}$ on $[0,v_h]$; from (\ref{OptP}), the maximizer for $R_{s2}$ on $[0,v_h]$ is $p^*(v_h)$. 
We have 
\begin{equation} \label{OptSpotRev}
    R_s^* := \max_{p_s\in[0,v_h]} R_s(p_s)= \max\{R_{s1}(p_1^*), \alpha R^*(v_h)\}. 
\end{equation}

Next we compare the maximum on the left and right intervals. 

If $p^*(v_h) \le v_l$, then $R_{s2}(p_s)$ is decreasing on $[v_l,v_h]$. Therefore, the optimal price is located on $[0,v_l]$. Since the optimizer of $R_p(p_s,v_l)$ is $p^*(v_l)<p^*(v_h)\le v_l$, the weighted average $R_s(p)$ is decreasing on $[v_l,v_h]$. We find the optimal price $p_1^*\in(0,v_l)$ by solving the first-order condition $R_{s1}'(p_s) = 0$.  

If $p^*(v_h) > v_l$, then we need to compare the maximum value on $[0, v_l]$ with the maximum value on $[v_l, v_h]$. When $p_1^* \ge v_l$, $R_{s1}$ is increasing on $[0,v_l]$ and the optimal price is located on the right interval and should be $p^*(v_h)$. We focus on the case when $p_1^*<v_l$.  We examine how the maximum revenues $R_{s1}(p_1^*)$ on the left interval and $\alpha R^*(v_h)$ on the right interval change with $\alpha$.  Using the envelope theorem, the derivative of $R_{s1}(p_1^*)$ with respect to $\alpha$ is given by $\frac{d R_{s1}(p_1^*)}{d \alpha} = \frac{\partial R_{s1}}{\partial \alpha}|_{p_s = p_1^*} = R_p(p_1^*, v_h) - R_p(p_1^*, v_l) > 0$, since $p_1^* < v_l < v_h$ and $R_p(p_1^*,v)$ is increasing in $v$.  On the right interval, $\frac{d (\alpha R^*(v_h))}{d \alpha} = R^*(v_h) = R_p(p^*(v_h), v_h) > R_p(p_1^*, v_h)$, since $p^*(v_h)$ maximizes $R_p(p, v_h)$.
It is clear that $\frac{d (\alpha R^*(v_h))}{d \alpha}> \frac{d R_{s1}^*}{d \alpha}$. Thus, as $\alpha$ increases from 0 to 1, $\alpha R^*(v_h)$ increases faster than $R_{s1}(p^*)$.  When $\alpha=0$, $R_{s1}(p^*) = R_p(p_l^*, v_l) >\alpha R^*(v_h)= 0$; when $\alpha=1$,  $\alpha R^*(v_h) = R_p(p^*(v_h), v_h)> R_{s1}(p^*)$.  Due to the monotonicity and continuity with respect to $\alpha$, there must exist a threshold $\hat{\alpha} \in (0, 1)$ such that when $\alpha = \hat{\alpha}$, $R_{s1}(p^*) = \alpha R^*(v_h)$. Note that $p^*(v_h)$ does not change with $\alpha$. For $\alpha<\hat{\alpha}$, $R_{s1}(p^*)>\alpha R^*(v_h)>R_s(v_l)$ (the second inequality is due to $p^*(v_h)>v_l$), implying $p_1^*<v_l$ and $p_1^*$ is the FOC point.  For $\alpha> \hat{\alpha}$, $\alpha R^*(v_h)>R_{s1}(p^*)$, the maximum on $(v_l,v_h]$ dominates the maximum on $[0,v_l]$ either at $p_1^*$ or the corner point $v_l$.  
\end{proof}

\paragraph{Advance Selling.}
Under advance selling,  purchase decisions are made before valuation uncertainty is resolved. 
Customers who purchase in advance may find the realized valuation $v_l$ is too low and do not want to join the queue. Solving $v_l - T_l W(T_l) = 0$, we obtain $T_l =\frac{\mu v_l}{\lambda v_l+1}$. Customers whose waiting cost is above $T_l$ would choose to balk if they purchase in advance and the realized value is $v_l$. We study two cases, illustrated in Figure~\ref{fig:AdvS}. 
\begin{figure}[hbtp]
 \centering
 \includegraphics[scale=0.3]{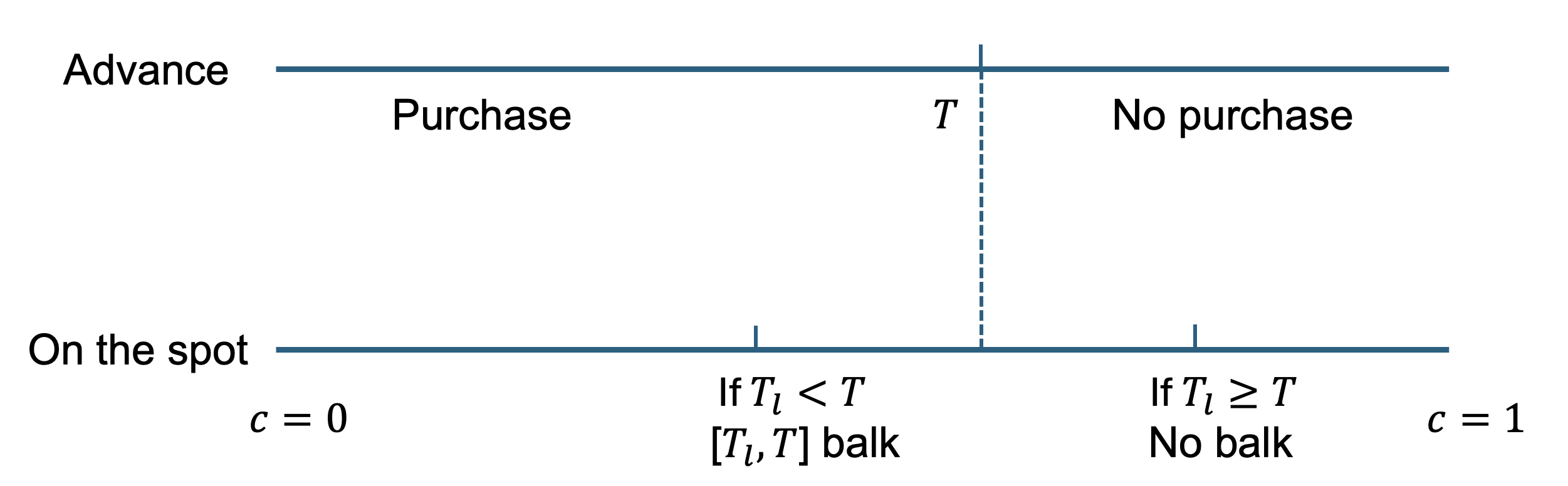}
\caption{Customer Behavior under Advance Selling}
\label{fig:AdvS}
\end{figure}

{\bf Case 1:} Customers who purchase in advance still join the queue in every realized state because their realized utility from joining remains nonnegative. In this case, the analysis becomes analogous to the base model, with the deterministic valuation $v$ replaced by the expected valuation $\bar{v}$. From (\ref{eq:adv_val}), $p_a = \bar{v} - T\,W(T)$. We write the expected revenue as a function of $T$
\[
R_T(T,\bar{v}) := \lambda\,(\bar{v} - T\,W(T))T, \quad 0\le T \le T_l. 
\]
$R_T$ is concave in $T$. We can solve the optimal cutoff and revenue
\begin{equation}\label{OptT}
   T_1^*=\frac{\mu}{\lambda}\bigg(1-\frac{1}{\sqrt{1+\lambda \bar{v}}}\bigg), \quad R_T(T,\bar{v}) = R^*(\bar{v}). 
\end{equation}
 If $T_1^* < T_l$, then the optimal revenue in this case is $R^*(\bar{v})$; otherwise, it is $R_T(T_l, \bar{v})$. 

{\bf Case 2:} Some customers who purchase in advance balk the queue. In this case, when considering whether to purchase the advance ticket, the utility of the marginal customer located at $T$ satisfies 
\[
\alpha (v_h - T W(T)) - p_a =0.  
\]
The expected revenue as a function of $T$ is 
\[
\lambda\, p_a T = \alpha \lambda\,(v_h - T\,W(T))T = \alpha R_T(T,v_h), \quad T > T_l.
\]
This revenue function is also concave in $T$. We can solve the optimal cutoff $T_2^*=\frac{\mu}{\lambda}\big(1-\frac{1}{\sqrt{1+\lambda v_h}}\big)$. If $T_2^*\ge T_l$, then the optimal revenue in this case is $\alpha R^*(v_h)$, otherwise, it is $\alpha R_a(T_l,v_h) = R_a(T_l,\bar{v})$. 

The maximum expected revenue under advance selling is 
\[
R_a(T)=\begin{cases}
    R_T(T,\bar{v}) & \quad \mbox{ if }\, 0\le T\le T_l \\
    \alpha R_T(T,v_h) & \quad \mbox{ if } \, T_l< T< \frac{\mu_h}{\lambda}
    \end{cases}
\]
$R_T(T,\bar{v})\ge \alpha R_T(T,v_h)$ is equivalent to $T\le T_l$. Therefore, we can extend the supports of both functions without affecting the maximum of $R_a$: 
\[
\max_{T\in\left[0,\frac{\mu_h}{\lambda}\right)} R_a(T) = \max_{T\in\left[0,\frac{\mu_h}{\lambda}\right)} \max\{R_{T}(T,\bar{v}), \alpha R_{T}(T,v_h)\} = \max\left\{\max_{T\in\left[0,\frac{\mu_h}{\lambda}\right)}R_{T}(T,\bar{v}),\max_{T\in\left[0,\frac{\mu_h}{\lambda}\right)}\alpha R_{T}(T,v_h)\right\}
\] 
From (\ref{OptT}), we have 
\begin{equation}\label{OptAdvRev}
    R_a^*:=\max_{T\in\left[0,\frac{\mu_h}{\lambda}\right)} R_a(T) = \max \{R^*(\bar{v}), \alpha R^*(v_h)\}. 
\end{equation}

Next we compare the maximum on the left and right intervals. Define $\theta: = 2v_l + \lambda v_l^2$, $T_1^* < T_l$ is equivalent to $\bar{v} < \theta$, and $T_2^* > T_l$ is equivalent to $v_h > \theta$. We have three scenarios:

1) if $\theta\le \bar{v}$, then $T_1^*\ge T_l$ and $T_2^*>T_l$. $R_a(T)$ is increasing in $T$ on $[0,T_l]$ and the global maximizer is $T_2^*$. 

2) if $\theta\ge v_h$, then $T_1^* < T_l$ and $T_2^*\le T_l$. $R_a(T)$ is decreasing in $T$ on $[T_l,1]$ and the global maximizer is $T_1^*$.  

3) if $\bar{v}<\theta<v_h$, then $T_1^*< T_l<T_2^*$. Then we need to compare the maximum on the left and right intervals. Solving $R^*(\bar{v}) = \alpha R^*(v_h)$ gives a trivial solution $\alpha =1$ and a non-trivial solution $\alpha = \tilde{\alpha}:=\frac{v_l^2(\lambda^2 v_l^2 + 4\lambda(v_h + v_l + \sqrt{1+\lambda v_h}v_l) + 8(1+\sqrt{1+\lambda v_h}))}{(4v_h-v_l(4+\lambda v_l))^2}$. When $0<\alpha<\tilde{\alpha}$, $R^*(\bar{v}) > \alpha R^*(v_h)$; when $\tilde{\alpha}<\alpha<1$, $R^*(\bar{v}) < \alpha R^*(v_h)$. Note $\tilde{\alpha}$ could be larger than 1. 
We have the following Lemma.
\begin{lemma}\label{lemma:OptAdv}
When service valuation follows the two-point distribution in (\ref{DisV}). Under advance selling, we have three scenarios: 1) if $\theta\le \bar{v}$, then the optimal cutoff is $T_2^*$ and the optimal revenue is $\alpha R^*(v_h)$; 2) if $\theta \ge v_h$, then the optimal cutoff is $T_1^*$ and the optimal revenue is $R^*(\bar{v})$; 3) if $\bar{v}<\theta<v_h$, when $0<\alpha \le\tilde{\alpha}$, the optimal cutoff is $T_1^*$, when $\tilde{\alpha}<\alpha<1$, the optimal cutoff is $T_2^*$. 
\end{lemma}

\textbf{Revenue Comparison.} 
We are ready to compare the maximum revenues of advance selling in (\ref{OptSpotRev}) and spot selling in (\ref{OptAdvRev}) under uncertain valuation.

\begin{prop} \label{prop:compare1}
When service valuation follows the two-point distribution in (\ref{DisV}), $R_s^*\le R_a^*$, where the equality is attained only if $p^*(v_h)\ge v_l$ and $\alpha>\hat{\alpha}$. 
\end{prop}
\begin{proof}
We only need to compare $R^*(\bar{v})$ with $R_{s1}(p_1^*)$. Since $R_p(p_s, v)$ is strictly concave in $v$ for any fixed $p$, and $\bar{v} = \alpha v_h + (1-\alpha)v_l,0 < \alpha < 1$, we have:
\[
R_{s1}(p) = \alpha R_p(p, v_h) + (1-\alpha) R_p(p, v_l) < R_p(p, \alpha v_h + (1-\alpha) v_l) = R_p(p, \bar{v})\le R^*(\bar{v}).
\]
The first inequality is from Jensen's inequality, and the second inequality is because $R^*(\bar{v})$ is the maximum revenue for the revenue function $R_p(p, \bar{v})$.  Therefore, for $\alpha\in(0,1)$, $R^*(\bar{v})$ dominates the maximum of $R_{s1}(p)$. That is, $R^*(\bar{v})>R_{s1}(p_1^*)$, which implies $R_s^* \le R_a^*$. The equality is achieved only when $R_a^* = R_s^* = \alpha R^*(v_h)$. Since $R^*(\bar{v})>R_{s1}(p_1^*)$ when $0<\alpha<1$, the set $\{\alpha: R_a^* = \alpha R^*(v_h)\}$ is a subset of $\{\alpha: R_s^* = \alpha R^*(v_h)\}$. From Lemma~\ref{lemma:OptSpot}, only when $p^*(v_h)\ge v_l$ and $\alpha>\hat{\alpha}$, we have $R_a^*=R_s^*$.   
\end{proof}

\textbf{Remark:}  Consider the setting without waiting costs. Suppose there is a unit mass of customers, and each customer's valuation follows the distribution in (\ref{DisV}).  Under advance selling, the optimal price is the expected valuation $\bar{v} = \alpha v_h + (1-\alpha)v_l$, and the expected revenue is also $\bar{v}$.  Under spot selling, however, with a single spot price, the optimal strategy is to set price $p_s = v_l$ if $v_l > \alpha v_h$ or $p_s = v_h$ otherwise. The expected revenue is $\max\{ \alpha v_h, v_l\} < \alpha v_h + (1-\alpha)v_l = \bar{v}$.  This special case also confirms advance selling is \emph{always} more profitable than spot selling when a single price is used in spot selling, which aligns with the findings of \cite{XieShugan2001}. The presence of waiting costs and queueing effects in our model does not alter this conclusion when the valuation is uncertain.

\subsection{Uncertain Capacity}
Now we consider the case where customer valuation is fixed at $v$, but the service capacity $\mu$ is uncertain and follows a two-point distribution:
\begin{equation} \label{MuDis}
   \mu = \begin{cases}
\mu_h, & \text{with probability } \alpha,\\[1mm]
\mu_l, & \text{with probability } 1-\alpha,
\end{cases} 
\end{equation}
where $\lambda \ge \mu_h>\mu_l$, since we focus on service systems with scarce capacity.  Let $\bar{\mu} = \alpha \mu_h + (1-\alpha) \mu_l$ be the expected capacity. 

\paragraph{Spot Selling.}
Under spot selling, the capacity is observed \emph{before} purchase.  The seller sets a spot price $p_s$. To streamline the notation, in this section, we reuse symbols from previous sections. From (\ref{eq:T}), we define $T(p,\mu):=\frac{\mu(v-p)}{\lambda(v-p)+1}$. When the capacity is observed to be high, the marginal customer type is $T(p_s, \mu_h)$.  When the capacity is low, the marginal customer type is $T(p_s, \mu_l)$.  Noting $T(p_s,\mu)$ is linear in $\mu$, the expected revenue under spot selling is:
\begin{equation*}
R_s(p_s) =\lambda\, p_s(\alpha T(p_s, \mu_h) + (1-\alpha)T(p_s, \mu_l)) =\lambda\, p_s T(p_s, \bar{\mu}). 
\end{equation*}
This revenue function is of the same form as in the base model (\ref{eq:RevD}), just with $\mu$ replaced by $\bar{\mu}$. Defining $R^*(\mu):= \frac{\mu(\sqrt{1 + \lambda v}-1)^2}{\lambda}$, we therefore have the following lemma. 
\begin{lemma}\label{lemma:OptSpot2}
When service capacity follows the two-point distribution in (\ref{MuDis}). Under spot selling, the optimal price is $p^*(v)$. The maximal revenue is $R_s^* = R^*(\bar{\mu})$. 
\end{lemma}

\paragraph{Advance Selling.}
In advance selling, the price $p_a$ is set before the capacity is revealed. Customers, when deciding whether to purchase, form expectations about waiting time since the realized capacity is unknown at that stage. Similar to the advance selling under the uncertain valuation case, customers who have purchased the ticket may balk the queue if the realized capacity turns out to be low, $\mu_l$.  Define $W(T,\mu):=\frac{1}{\mu-\lambda T}$.  Solving the indifference condition $v-TW(T,\mu_l)=0$ yields the balking cutoff $T_l := \frac{\mu_l v}{1+\lambda v}$. If the purchase cutoff satisfies $T\le T_l$, no customer balks the queue. In this case, the expected waiting time is: $\alpha\,W(T,\mu_h) + (1-\alpha)\,W(T,\mu_l)$. The expected utility of the marginal customers at cutoff $T$ satisfies $v - p_a - T\,(\alpha\,W(T,\mu_h) + (1-\alpha)\,W(T,\mu_l))=0$. Define the revenue function $R_T(T,\mu): = \lambda(v-T W(T,\mu))T$. Then, the revenue under advance selling becomes
\[
R_{a1}(T)= \lambda (v - T\,(\alpha W(T,\mu_h) + (1-\alpha) W(T,\mu_l))) T = \alpha R_T(T,\mu_h) + (1-\alpha) R_T(T,\mu_l). 
\]
Since $R_T(T,\mu)$ is concave in $T$, the weighted sum $R_{a1}(T)$ is also concave in $T$. Let $T_1^*$ denote the maximizer of $R_{a1}(T)$, obtained from its first-order condition. 

If the purchase cutoff satisfies $T>T_l$, the customers whose type is on $(T_l, T]$ would balk the queue when they observe the low capacity $\mu_l$. The expected utility of customers at cutoff $T$ satisfies $\alpha(v - T\,W(T,\mu_h))-p_a=0$. Then, the revenue under advance selling is given by:
\[
\lambda p_a T = \lambda \alpha(v - T\,W(T,\mu_h)) T = \alpha R_T(T,\mu_h). 
\]
The maximum expected revenue under advance selling is 
\[
R_a(T)=\begin{cases}
    R_{a1}(T) & \quad \mbox{ if }\, 0\le T\le T_l \\
    \alpha R_T(T,\mu_h) & \quad \mbox{ if } \, T_l< T\le \frac{\mu_h}{\lambda}
    \end{cases}
\]
Since $R_T(T,\mu_l)\ge0$ when $0\le T\le T_l$ and $R_T(T,\mu_l)<0$ when $T_l< T<1$, $R_{a1}(T)\ge \alpha R_T(T,\mu_h)$ when $0\le T\le T_l$ and $R_{a1}(T)< \alpha R_T(T,\mu_h)$ when $T_l<T\le 1$. Therefore, we can extend the supports of both functions without affecting the maximum of $R_a$: 
\[
\max_{T\in\left[0,\frac{\mu_h}{\lambda}\right)} R_a(T) = \max_{T\in\left[0,\frac{\mu_h}{\lambda}\right)} \max\{R_{a1}(T), \alpha R_{T}(T,\mu_h)\} = \max\left\{\max_{T\in\left[0,\frac{\mu_h}{\lambda}\right)}R_{a1}(T),\max_{T\in\left[0,\frac{\mu_h}{\lambda}\right)}\alpha R_{T}(T,\mu_h)\right\}
\] 
Solving the first-order condition of $R_T(T,\mu_h)$ gives $T_h^*:= \frac{\mu_h}{\lambda}\big(1-\frac{1}{\sqrt{1+\lambda v}}\big)$. 
\begin{equation*}\label{OptAdvRev2}
    R_a^*:=\max_{T\in\left[0,\frac{\mu_h}{\lambda}\right)} R_a(T) = \max \{R_{a1}(T_1^*), \alpha R^*(\mu_h)\}. 
\end{equation*}
Lemma~\ref{lemma:OptAdv2} establishes the optimal threshold and revenue under advance selling. 
\begin{lemma}\label{lemma:OptAdv2}
When service capacity follows the two-point distribution in (\ref{MuDis}). Under advance selling, if $T_h^* \ge T_l$, then there exists a threshold probability $\alpha_T\in(0,1)$ such that if $\alpha>\alpha_T$, the optimal revenue is achieved at threshold $T_h^*$. In all other cases, the optimal revenue is achieved at threshold $T_1^*\in(0,T_l)$, where $T_1^*$ solves the first-order condition for $R_{a1}(T)$. 
\end{lemma}
\begin{proof}
We compare the maximum on the left and right intervals. 

If $T_h^*\le T_l$, then $\alpha R_{T}(T,\mu_h)$ is decreasing on $[T_l,1]$. Therefore, the optimal threshold is located on $[0,T_l]$. The first-order condition yields the maximizer of $R_T(T,\mu_l)$ is $\frac{\mu_l}{\lambda}\big(1-\frac{1}{\sqrt{1+\lambda v}}\big)$, which is also below $T_l$. Therefore, the weighted average $R_{a1}(T)$ is decreasing on $[T_l,1]$. We find the optimal threshold $T_1^*\in(0,T_l)$ by solving the first-order condition $R_{a1}'(T) = 0$. 

If $T_h^*> T_l$, then we need to compare the maximum value on $[0, T_l]$ with the maximum value on $[T_l, T_h]$. When $T_1^* \ge T_l$, $R_{a1}(T)$ is increasing on $[0,T_l]$ and the optimal threshold is located on the right interval and should be $T_h^*$. We focus on the case when $T_1^*<T_l$.  We examine how the maximum revenues $R_{a1}(T_1^*)$ on the left interval and $\alpha R^*(\mu_h)$ on the right interval change with $\alpha$.  Using the envelope theorem, the derivative of $R_{a1}(T_1^*)$ with respect to $\alpha$ is given by $\frac{d R_{a1}(T_1^*)}{d \alpha} = \frac{\partial R_{a1}}{\partial \alpha}|_{T = T_1^*} = R_T(T_1^*, \mu_h) - R_T(T_1^*, \mu_l) =\frac{\lambda T_1^{*2}(\mu_h-\mu_l)}{(\mu_h-\lambda T_1^*)(\mu_l-\lambda T_1^*)}> 0$.  On the right interval, $\frac{d (\alpha R^*(\mu_h))}{d \alpha} = R^*(\mu_h) = R_T(T_h^*, \mu_h) > R_T(T_1^*, \mu_h)$, since $T_h^*$ maximizes $R_T(T, \mu_h)$.
As long as $T_1^*<T_l$, we have $R_T(T_1^*,\mu_l)>R_T(T_l,\mu_l)=0$, implying $\frac{d (\alpha R^*(\mu_h))}{d \alpha}> \frac{d R_{a1}^*}{d \alpha}$. Thus, as $\alpha$ increases from 0 to 1, $\alpha R^*(\mu_h)$ increases faster than $R_{a1}(T_1^*)$. Meanwhile, $T_1^*$ is increasing with $\alpha$.  When $\alpha=0$, $R_{a1}(T_1^*) = R_T(T_1^*, \mu_l) >\alpha R^*(\mu_h)= 0$; when $\alpha = 1$, $T_1^* = T_h^*>T_l$. There exists a value $\alpha_T$ such that $T_1^*=T_l$ when $\alpha=\alpha_T$. Due to the monotonicity and continuity with respect to $\alpha$, there must exist a threshold $\hat{\alpha} \in (0,\alpha_T)$ such that when $\alpha = \hat{\alpha}$, $R_{a1}(T_1^*) = \alpha R^*(\mu_h)$. Note that $T_h^*$ does not change with $\alpha$. For $\alpha\in[0,\hat{\alpha})$, $R_{a1}(T_1^*)>\alpha R^*(\mu_h)$; for $\alpha\in (\hat{\alpha},\alpha_T)$, $\alpha R^*(\mu_h)>R_{a1}(T_1^*)$; for $\alpha\in[\alpha_T,\mu_h/\lambda]$, $T_1^*\ge T_l$ and the maximum is achieved on the right interval at $T_h^*$.  
\end{proof}

We now compare the profitability of advance and spot selling under uncertain capacity.

\begin{prop}\label{prop:Com2}
When service capacity $\mu$ follows the two-point distribution in (\ref{MuDis}), the optimal advance selling revenue is always less than the optimal spot selling revenue when $\alpha\in(0,1)$.
\end{prop}
\begin{proof}
If $R_a^* = \alpha R^*(\mu_h)$, then the threshold is $T_h^*$ and the advance price is $p^*(v)$. By charging the same price, the spot selling obtains a revenue of $R_s(p^*(v)) = \alpha R^*(\mu_h) + (1-\alpha) R^*(\mu_l) > R_a^*$.

If $R_a^* = R_{a1}(T_1^*)$, then the threshold $T_1^*$ satisfies 
\[
v - p_a = \alpha T_1^* W(T_1^*, \mu_h) + (1-\alpha) T_1^* W(T_1^*,\mu_l) = \alpha C(\mu_h)+(1-\alpha) C(\mu_l), 
\]
where $C(\mu):=\frac{T_1^*}{\mu-\lambda T_1^*}$. The second derivative of $C(\mu)$ with respect to $\mu$ is $\frac{d^2 C(\mu)}{d\mu^2} = \frac{2T_1^*}{(\mu-\lambda T_1^*)^3} > 0$ for $\mu > \lambda T_1^*$. Applying Jensen's inequality for the convex function $C(\mu)$, when $\alpha\in(0,1)$, 
\[
v-p_a = \alpha C(\mu_h) + (1-\alpha )C(\mu_l) > C(\alpha \mu_h + (1-\alpha) \mu_l) = C(\bar{\mu}).
\]
When spot selling charges the same $p_a$, solving $v-p_a =\frac{\tilde{T}}{\bar{\mu}-\lambda\Tilde{T}}$, we obtain the corresponding threshold $\tilde{T} = \frac{\bar{\mu}(v-p_a)}{\lambda(v-p_a)+1} = T(p_a, \bar{\mu})$. Since $\frac{x}{\bar{\mu}-\lambda x}$ is increasing in $x$ for $x\in[0,1]$, and $v-p_a > \frac{T_1^*}{\bar{\mu}-\lambda T_1^*}$, we must have $T_1^* < \tilde{T}$. 
Then $R_s(p_a) = \lambda\,p_a \tilde{T} > \lambda\,p_a T_1^* = R_a^*$. 
\end{proof}

\textbf{Remark:}  When capacity is uncertain, the expected waiting time under advance selling varies significantly across different states, as demand remains fixed regardless of the realized capacity. In contrast, spot selling allows customers to adjust their effective demand in response to the realized capacity, thereby reducing the variance of the expected waiting time. On the other hand, when customer valuation is uncertain, the expected waiting time under advance selling is not state-dependent, because demand is locked in advance. However, under spot selling, the expected waiting time fluctuates with the state, as demand adjusts dynamically based on realized valuations. In congestion-sensitive systems under uncertainty, reducing the variance of the expected waiting time appears to be beneficial for revenue maximization.

\section{Robustness Check under Continuous Distributions}
In this section, we study continuous distributions of valuation and service capacity to examine the robustness of the main revenue-dominance insights. The two-point models in Section 2 are useful because they provide explicit high- and low-state comparisons and make the economic mechanisms transparent. At the same time, the two-point discrete distributions with atoms are not special cases of the continuous distributions considered in this section. Thus, the continuous-distribution analysis should be viewed as a robustness check rather than as generalization for the two-point analysis. The results below show that the main insights are not driven solely by the binary-state specification. Throughout this section, we maintain the capacity-scarcity assumption and the uniform waiting-cost distribution on \([0,1]\), while allowing valuation and capacity uncertainty to follow continuous distributions.


\subsection{Uncertain Valuation}

We assume customer valuation follows a general continuous distribution on support $[v_l, v_h]$, where $0\le v_l \le v_h$. The distribution function is $F(v)$. 

\paragraph{Spot Selling.} We follow the notation in Section 2.1. Under a given spot price $p_s$, the expected revenue rate under spot selling is:
\[
R_s(p_s)  = \int_{v_l}^{v_h} [R_p(p_s, v)]^+ d F(v).
\]
Since $R_p(p_s,v)<0$ when $v<p_s$ and $R_p(p_s,v)\ge 0$ when $v\ge p_s$. We have
\[
R_s(p_s) =   \int_{v\ge p_s} R_p(p_s, v) d F(v) =  \lambda \mu\int_{v\ge p_s}\frac{p_s\,(v-p_s)}{1+\lambda\,(v-p_s)}\,dF(v).
\]

\paragraph{Advance Selling.} Suppose the purchase cutoff is $T$, then define $U(T):=TW(T)$. When realized valuation $v\ge U(T)$, no one balks the queue; otherwise, the marginal customer at $T$ would balk the queue. The expected utility of the marginal customer purchasing the advance ticket at price $p_a$ is 
\[
\int_{v\ge U(T)} (v-U(T)) \, d F(v) - p_a = 0. 
\]
The expected revenue rate is 
\[
R_a(T) = \lambda\,p_a\, T =\lambda\, T\int_{v\ge U(T)} (v-U(T))\, d F(v). 
\]
Writing $u:=U(T)=\frac{T}{\mu-\lambda T}\in [0,\infty)$ so that $T=\frac{\mu u}{1+\lambda u}$, the expected revenue is
\[
R_a(u)\;=\;\lambda\,\mu\,\frac{u}{1+\lambda u}\int_{v\ge u}(v-u)\,dF(v).
\]

\begin{prop}
Assume customer valuation follows a continuous distribution $F$ on $[v_\ell,v_h]\subseteq[0,\infty)$. Advance selling weakly dominates spot selling in revenue for any continuous $F$.   That is, 
\[
\max_{u\ge 0}R_a(u)\;\ge\;\max_{p\ge 0}R_s(p).
\]
\end{prop}

\begin{proof}
We rewrite the spot revenue using a change of variables. 
Let $y = v - p$, so $v = y + p$. Then $v \geq p$ becomes $y \geq 0$, and:
$$R_s(p) = \lambda\,\mu \int_{y \geq 0} \frac{p \, y}{1+\lambda y} \, dF(y + p)$$
Consider the function $g(y) := \frac{p \, y}{1+\lambda y}$ for $y \geq 0$. We seek constants $A, B \geq 0$ such that:
$$A(y + B) \geq \frac{py}{1+\lambda y} \quad \text{for all } y \geq 0$$
The best choice is when this linear function is tangent to the concave function $g(y) = \frac{py}{1+\lambda y}$ at some point $y_0$. The function $g(y)$ has derivative: $g'(y) = \frac{p}{(1+\lambda y)^2}$. 
For the tangent at $y_0$:
\[
A(y + B) = g(y_0) + g'(y_0)(y - y_0)
\]
So: $A = g'(y_0) = \frac{p}{(1+\lambda y_0)^2}$ and $B = g(y_0)/A- y_0=\lambda y_0^2$. 
A natural choice of $y_0$ is to match $A$ with $\frac{u}{1+\lambda u}$ corresponding to the advance selling revenue format.  Setting $u=y_0$ and solving $\frac{p}{(1+\lambda y_0)^2} = \frac{y_0}{1+\lambda y_0}$ gives $p = y_0 (1+\lambda y_0)$. Solving this quadratic function gives the positive root: $y_0 = \frac{\sqrt{1+\lambda p} - 1}{2\lambda}$. 
With this choice, we have 
\[
g(y) \le A(y+B) = A(y + \lambda y_0^2) = A(y + p-y_0) =\frac{y_0}{1+\lambda y_0}(y + p-y_0). 
\]
Therefore, 
\begin{align*}
R_s(p) &= \lambda\,\mu \int_{y \geq 0} \frac{py}{1+\lambda y} \, dF(y+p) \leq \lambda\,\mu \int_{y \geq 0} \frac{y_0}{1+\lambda y_0}(y + p - y_0) \, dF(y+p) \\
&= \lambda\,\mu \int_{v \geq p} \frac{y_0}{1+\lambda y_0}(v - y_0) dF(v)\le \lambda\,\mu \int_{v \geq y_0} \frac{y_0}{1+\lambda y_0}(v - y_0) dF(v)  =R_a(y_0). 
\end{align*}
The second inequality is due to $p = y_0 + \lambda y_0^2\ge y_0$ and the integrand is non-negative. Let $p_s^*$ be the maximizer for $R_s(p)$ and $y_0^*=\frac{\sqrt{1+\lambda p_s^*} - 1}{2\lambda}$. We have 
\[
\max_{u\ge 0}R_a(u) \ge R_a(y_0^*)\ge R_s(p_s^*)=\max_{p\ge 0}R_s(p).
\]
\end{proof}

\subsection{Uncertain Capacity}
Now fix the valuation $v$. Let the service capacity $\mu$ follow a continuous distribution $H$ supported on $[\mu_l,\mu_h]$, where $\lambda\ge\mu_h>\mu_l$. 

\paragraph{Spot Selling.}
The seller sets a spot price $p_s$. As in Section 2.2, when the capacity is observed to be $\mu$,  the purchase cutoff is $T(p_s,\mu)=\frac{\mu(v-p_s)}{\lambda(v-p_s)+1}$. The expected revenue under spot selling is:
\begin{equation*}
R_s(p_s) = \lambda \int p_s T(p_s, \mu)\, dH(\mu) = \lambda\,p_s \frac{(v-p_s)}{\lambda(v-p_s)+1} \int \mu dH(\mu) = \lambda\,R_p(p_s,\bar{\mu}),  
\end{equation*}
where $\bar{\mu}$ is the expectation of the service capacity. We have the following lemma similar to Lemma~\ref{lemma:OptSpot2}. 
\begin{lemma}\label{lemma:OptSpot3}
Assume the service capacity follows a continuous distribution $H$ on $[\mu_l,\mu_h]\subseteq[0,\lambda]$. Under spot selling, the optimal price is $p^*(v)$. The maximal revenue is $R_s^* = R^*(\bar{\mu})$. 
\end{lemma}

\paragraph{Advance Selling.} 
Suppose the purchase cutoff is $T$. Following Section 2.2, we have defined $C(\mu)=\frac{T}{\mu-\lambda T}$. When realized capacity satisfies $C(\mu)\le v$, or $\mu\ge \frac{T(1+\lambda v)}{v}$, no one balks the queue; otherwise, the marginal customer at $T$ would balk the queue.  The expected utility of the marginal customer purchasing the advance ticket at price $p_a$ is 
\[
\int_{\mu\ge \frac{T(1+\lambda v)}{v}} (v- C(\mu)) \, d H(\mu) - p_a = 0. 
\]
The expected revenue rate is 
\[
R_a(T) = \lambda\,p_a\, T = \lambda\,T\int_{\mu\ge \frac{T(1+\lambda v)}{v}} (v-C(\mu))\, d H(\mu). 
\]
Suppose $T^*$ is the maximizer of $R_a(T)$, and the corresponding optimal advance price is $p_a^*$. Since $C(\mu)$ is convex, applying Jensen's inequality we have 
\[
p_a^* = \int_{\mu\ge \frac{T^*(1+\lambda v)}{v}} (v- C(\mu)) \, d H(\mu)
\le v - C\bigg(\int_{\mu\ge \frac{T^*(1+\lambda v)}{v}} \mu \, dH(\mu) \bigg) \le v - C(\bar{\mu}).  
\]
The second inequality is due to $C(\mu)$ is decreasing and $\int_{\mu\ge \frac{T^*(1+\lambda v)}{v}} \mu \, dH(\mu)\le \bar{\mu}$. 
Suppose the spot selling also adopt price $p_a^*$, then the cutoff $\tilde{T}$ satisfies $v-p_a^* - \frac{\tilde{T}}{\bar{\mu}-\lambda \tilde{T}} = 0$. We have
\[
\frac{\tilde{T}}{\bar{\mu}-\lambda \tilde{T}} \ge \frac{T^*}{\bar{\mu}-\lambda T^*} \Longleftrightarrow \tilde{T} \ge T^*. 
\]
Therefore, 
\[
\max_{T\in\left[0,\frac{\mu_h}{\lambda}\right)}R_a(T) = \lambda\, p_a^* \,T^* \le \lambda\,p_a^* \tilde{T} \le  \max_{p\ge 0}R_s(p) = R^*(\bar{\mu}).
\]
We have the following dominance result. 
\begin{prop}
Assume service capacity follows a continuous distribution $H$ on $[\mu_l,\mu_h]\subseteq[0,\lambda]$. Spot selling weakly dominates advance selling in revenue for any continuous $H$.   That is, 
\[
\max_{T\in\left[0,\frac{\mu_h}{\lambda}\right)}R_a(T)\;\le\;\max_{p\ge 0}R_s(p).
\]
\end{prop}

\subsection{When Both Uncertainty Sources Exist}

We now consider the case in which both sources of uncertainty are present. Customer valuation follows a continuous distribution \(F\) on \([v_l,v_h]\), and service capacity follows a continuous distribution \(H\) on \([\mu_l,\mu_h]\). The two sources of uncertainty are assumed to be independent.

\paragraph{Spot Selling.} Similar to Section 3.1, under a given spot price $p_s$, the expected revenue rate is 
\begin{align*}
    R_s(p_s) &=  \int_{\mu} \int_{v\ge p_s} R_p(p_s,v) \,dF(v)\, dH(\mu) 
 =\lambda\, \int_{\mu}  \mu\int_{v\ge p_s}\frac{p_s\,(v-p_s)}{1+\lambda\,(v-p_s)}\,dF(v)\,dH(\mu) \\
 &= \lambda\,p_s\,\bar{\mu} \int_{v\ge p_s}\frac{(v-p_s)}{1+\lambda\,(v-p_s)}\,dF(v). 
\end{align*}
Note that $\bar{\mu}$ becomes a coefficient of the objective and the maximizer $p_s^*$ only depends on valuation distribution and does not depend on the capacity distribution. 

\paragraph{Advance Selling.} Similar to Section 3.1, under a given threshold $T$, the advance price and the expected revenue rate are 
\[
p_a(T) = \int_{\mu>\lambda T} \int_{v\ge \frac{T}{\mu-\lambda T}} \bigg(v- \frac{T}{\mu-\lambda T}\bigg) \, d F(v)\,dH(\mu), \quad R_a(T) =\lambda\, T p_a(T). 
\]

\paragraph{Numerical Design.}
We set
\[
\lambda=2,
\qquad
v\sim U[\bar v-d_v,\bar v+d_v],
\qquad
\bar v=1,
\]
and
\[
\mu\sim U[\bar\mu-d_\mu,\bar\mu+d_\mu],
\qquad
\bar\mu=1.
\]
We vary
$d_v\in\{0.05,0.15,0.25,0.35,0.45\}$,  and $d_\mu\in\{0.05,0.15,0.25,0.35,0.45\}$. 
For all parameter combinations, \(v_l>0\) and $\mu_h=1+d_\mu\le 1.45<2=\lambda$. 

\paragraph{Example: $d_v=0.45,d_{\mu}=0.45$.}
Under spot selling,
\[
R_s(p_s)
=
2p_s\cdot \frac{1}{0.9}
\int_{\max\{p_s,0.55\}}^{1.45}
\frac{v-p_s}{1+2(v-p_s)}\,dv.
\]
Numerical maximization gives
\[
p_s^*\approx 0.5779,
\qquad
R_s(p_s^*)\approx 0.23588.
\]

Under advance selling,
\[
R_a(T)
=
2T\cdot \frac{1}{0.9}\cdot\frac{1}{0.9}
\int_{0.55}^{1.45}
\int_{0.55}^{1.45}
\left(
v-\frac{T}{\mu-2T}
\right)^+
\mathbf 1_{\{\mu>2T\}}\,dv
\,d\mu .
\]
Numerical maximization over \(T\in[0,1.45/2)\) gives
\[
T_a^*\approx 0.1849,
\qquad
R_a(T_a^*)\approx 0.23566.
\]

Therefore, the revenue difference  $\Delta R:=R_a(T_a^*)-R_s(p_s^*)
\approx
0.23566-0.23588=
-0.00023$. In this example, spot selling slightly outperforms advance selling.

For each pair \((d_v,d_\mu)\), we numerically optimize \(R_s(p_s)\) over \(p_s\in[0,v_h]\) and \(R_a(T)\) over \(T\in[0,\mu_h/\lambda)\). The numerical integration is performed by Gaussian quadrature, and the one-dimensional optimization problems are solved by bounded scalar optimization. Table~\ref{tab:dispersion_revenue_lambda_large_vbar1} reports the revenue difference $\Delta R$. 
A positive value means that advance selling generates higher revenue, whereas a negative value (colored in red) means that spot selling generates higher revenue.

\begin{table}[ht]
\centering
\caption{Revenue Difference \(\Delta R\) under Both Valuation and Capacity Uncertainty}
\label{tab:dispersion_revenue_lambda_large_vbar1}
\begin{tabular}{c|ccccc}
\hline
 & \(d_\mu=0.05\) & \(d_\mu=0.15\) & \(d_\mu=0.25\) & \(d_\mu=0.35\) & \(d_\mu=0.45\) \\
\hline
\(d_v=0.05\) &  0.00002 & {\color{red}-0.00310} & {\color{red}-0.00947} & {\color{red}-0.01938} & {\color{red}-0.03331} \\
\(d_v=0.15\) &  0.00327 &  0.00015 & {\color{red} -0.00622} & {\color{red}-0.01613} & {\color{red} -0.03003} \\
\(d_v=0.25\) &  0.00972 &  0.00660 &  0.00023 & {\color{red}-0.00967} & {\color{red}-0.02337} \\
\(d_v=0.35\) &  0.01931 &  0.01619 &  0.00982 & {\color{red}-0.00006} & {\color{red} -0.01334} \\
\(d_v=0.45\) &  0.03168 &  0.02856 &  0.02219 &  0.01254 & {\color{red}-0.00023} \\
\hline
\end{tabular}
\end{table}

Table~\ref{tab:dispersion_revenue_lambda_large_vbar1} shows a clear pattern. Holding capacity dispersion fixed, \(\Delta R\) increases as valuation dispersion \(d_v\) increases. This is consistent with the single-source result that valuation uncertainty favors advance selling. Holding valuation dispersion fixed, \(\Delta R\) decreases as capacity dispersion \(d_\mu\) increases, consistent with the single-source result that capacity uncertainty favors spot selling.

The table also shows that the revenue ranking can switch depending on the relative dispersion of the two uncertainty sources. When valuation dispersion is small and capacity dispersion is large, spot selling is more profitable. In contrast, when valuation dispersion is large and capacity dispersion is relatively small, advance selling is more profitable. 
The boundary between the two regimes shifts with the relative magnitude of the two dispersions. These numerical results suggest that the two single-source forces continue to operate when both sources of uncertainty coexist: valuation dispersion favors advance selling, whereas capacity dispersion favors spot selling. 

\subsection{Numerical Comparison for Social Welfare and Consumer Surplus}
Now we consider social welfare and consumer surplus under the two pricing strategies in the dual-uncertainty setting of Section 3.3. For a realized state \((v,\mu)\) in the unobservable \(M/M/1\) queue, the joining cutoff is $T(v,\mu)=\frac{\mu(v-p_s)}{\lambda(v-p_s)+1}$ under spot selling with price \(p_s\), and $T(v,\mu)=\min\left\{T_a,\frac{\mu v}{\lambda v+1}\right\}$ 
under advance selling with purchase cutoff \(T_a\). Note that under advance selling, if \(T_a>T(v,\mu)\), then customers whose waiting-cost types lie in \((T(v,\mu),T_a]\) will balk after the uncertainty is resolved.
The steady-state waiting time is $W(T,\mu)=\frac{1}{\mu-\lambda T}$. The corresponding social welfare rate (total valuation of customers minus their aggregate waiting cost per unit time) is
\begin{align*}
{\cal W}(v,\mu)
\;&=\;\lambda\!\left[\,v\,T(v,\mu)-\int_0^{T(v,\mu)} c\,W(T(v,\mu),\mu)dc\,\right]
\;=\;\lambda\!\left[\,v\,T(v,\mu)-\frac{T(v,\mu)^2}{2(\mu-\lambda T(v,\mu))}\,\right].
\end{align*}
Integrating over $v$ and $\mu$ gives social welfare rate
\[
SW = \int_\mu\int_{v} {\cal W}(v,\mu) d F(v) dH(\mu). 
\]
The consumer surplus rate is
\[
CS = SW - R, 
\]
where the revenue $R =R_s(p_s)$ under spot selling, and $R=R_a(T_a)$ under advance selling. 

We use the same numerical setting as in Section~3.3:
\[
\lambda=2,
\qquad
v\sim U[1-d_v,1+d_v],
\qquad
\mu\sim U[1-d_\mu,1+d_\mu].
\]
Table~\ref{tab:welfare_surplus_lambda_large_vbar1} reports four representative cases,  combining low or high valuation dispersion with low or high capacity dispersion.

\begin{table}[ht]
\centering
\caption{Revenue, Social Welfare, and Consumer Surplus under Dispersion Combinations}
\label{tab:welfare_surplus_lambda_large_vbar1}
\begin{tabular}{cc|cc|cc|cc}
\hline
\(d_v\) & \(d_\mu\)
& \(R_s^*\) & \(R_a^*\)
& \(SW_s\) & \(SW_a\)
& \(CS_s\) & \(CS_a\) \\
\hline
0.05 & 0.05 & 0.26754 & 0.26756 & 0.34535 & 0.34474 & 0.07781 & 0.07717 \\
0.05 & 0.45 & 0.26754 & 0.23423 & 0.34535 & 0.29632 & 0.07781 & 0.06209 \\
0.45 & 0.05 & 0.23588 & 0.26756 & 0.34509 & 0.34474 & 0.10921 & 0.07717 \\
0.45 & 0.45 & 0.23588 & 0.23566 & 0.34509 & 0.30244 & 0.10921 & 0.06678 \\
\hline
\end{tabular}
\end{table}

Several observations emerge from Table~\ref{tab:welfare_surplus_lambda_large_vbar1}. First, advance selling extracts more surplus from customers through ex-ante commitment, whereas spot selling leaves more surplus to customers. Second, when valuation dispersion is low and capacity dispersion is high, spot selling dominates advance selling in revenue, welfare, and consumer surplus. 
Third, when valuation dispersion is high and capacity dispersion is low, spot selling generates lower revenue but may generate higher social welfare than advance selling. These numerical results reinforce that the revenue-maximizing selling regime need not coincide with the welfare-maximizing or consumer-surplus-maximizing regime.

\section{Conclusion}
In this paper, we have developed and rigorously analyzed a queueing-based model comparing advance selling and spot selling strategies under two distinct sources of uncertainty: uncertain customer valuation from the demand side, uncertain service capacity from the supply side.  Our analysis yields nuanced insights into optimal pricing strategies in congested service systems.

Under conditions of uncertain customer valuation---a scenario often driven by exogenous factors such as weather conditions that affect service appeal---our findings indicate that advance selling enables the service provider to effectively ``lock in'' commitments from customers, particularly when the customer density is high enough, or the probability of realizing a high valuation state is relatively low. In such cases, advance selling becomes the more profitable strategy. Conversely, when uncertainty originates from fluctuations in service capacity, as commonly observed in gig economies and other dynamic service contexts, our analysis demonstrates that the flexibility to allow customers to purchase after the uncertainty is resolved---i.e., spot selling---consistently leads to higher revenue.  This is attributed to the ability to adapt demand to the realized state of capacity, mitigating the risk of long waiting times compared with locking in demand in advance. 

Our results highlight the importance for service providers to diagnose and understand the primary source of uncertainty in their operating environment. The optimal choice between advance and spot selling is not universal but rather fundamentally contingent on whether the dominant uncertainty lies with customer valuations or with supply-side capacity. These insights offer actionable guidance for service managers in designing and implementing effective pricing strategies in the face of real-world uncertainties inherent in service operations. 

Several modeling assumptions delimit the scope of our results. First, the uniform waiting-cost distribution is a normalized benchmark that allows closed-form analysis; other distributions may be more suitable in applications with heavy-tailed delay sensitivity. Second, we assume that waiting costs are independent of valuation and capacity states. Individual-level correlation between valuation and waiting cost may affect the composition of participating customers and would be an important extension. Third, the stochastic service rate in the M/M/1 model should be interpreted as a reduced-form representation of effective capacity uncertainty rather than a literal model of all multi-server capacity fluctuations.

Further research could extend this model to more complex queueing systems, such as multi-server queues or priority queues, to explore how these structures influence the effectiveness of advance and spot selling. Additionally, it would be interesting to investigate two-stage pricing strategies, where the seller sets both an advance price and a spot price, allowing for more dynamic revenue optimization. Finally, a more realistic extension could consider scenarios where customers arrive in two stages, enabling a richer analysis of demand dynamics and pricing strategies.

\bibliographystyle{plainnat}
\bibliography{ref}

\end{document}